\documentclass[a4paper]{amsart}

\RequirePackage{amsmath} \RequirePackage{amssymb}
\usepackage{amscd,latexsym,amsthm,amsfonts,amssymb,amsmath,amsxtra}
\usepackage[colorlinks=true,urlcolor=blue,citecolor=blue]{hyperref}
\usepackage[export]{adjustbox}
\usepackage[all]{xy}
\usepackage[OT2,T1]{fontenc}

\DeclareSymbolFont{cyrletters}{OT2}{wncyr}{m}{n}
\DeclareMathSymbol{\Sha}{\mathalpha}{cyrletters}{"58}

\newcommand{\floor}[1]{{\left\lfloor#1\right\rfloor}}
\newcommand{\round}[1]{{\left\lceil#1\right\rfloor}}

\begin{document}

    \theoremstyle{plain}
    \newtheorem{thm}{Theorem} \newtheorem{cor}[thm]{Corollary}
    \newtheorem{lemma}[thm]{Lemma}  \newtheorem{prop}[thm]{Proposition}
    \newtheorem{conj}[thm]{Conjecture}  \newtheorem{fact}[thm]{Fact}
    \newtheorem{claim}[thm]{Claim}
    \theoremstyle{definition}
    \newtheorem{defn}[thm]{Definition}
    \newtheorem{example}[thm]{Example}
    \newtheorem{exercise}[thm]{Exercise}
    \theoremstyle{remark}
    \newtheorem*{remark}{Remark}

    \newcommand{\BA}{{\mathbb {A}}} \newcommand{\BB}{{\mathbb {B}}}
    \newcommand{\BC}{{\mathbb {C}}} \newcommand{\BD}{{\mathbb {D}}}
    \newcommand{\BE}{{\mathbb {E}}} \newcommand{\BF}{{\mathbb {F}}}
    \newcommand{\BG}{{\mathbb {G}}} \newcommand{\BH}{{\mathbb {H}}}
    \newcommand{\BI}{{\mathbb {I}}} \newcommand{\BJ}{{\mathbb {J}}}
    \newcommand{\BK}{{\mathbb {K}}} \newcommand{\BL}{{\mathbb {L}}}
    \newcommand{\BM}{{\mathbb {M}}} \newcommand{\BN}{{\mathbb {N}}}
    \newcommand{\BO}{{\mathbb {O}}} \newcommand{\BP}{{\mathbb {P}}}
    \newcommand{\BQ}{{\mathbb {Q}}} \newcommand{\BR}{{\mathbb {R}}}
    \newcommand{\BS}{{\mathbb {S}}} \newcommand{\BT}{{\mathbb {T}}}
    \newcommand{\BU}{{\mathbb {U}}} \newcommand{\BV}{{\mathbb {V}}}
    \newcommand{\BW}{{\mathbb {W}}} \newcommand{\BX}{{\mathbb {X}}}
    \newcommand{\BY}{{\mathbb {Y}}} \newcommand{\BZ}{{\mathbb {Z}}}

    \newcommand{\CA}{{\mathcal {A}}} \newcommand{\CB}{{\mathcal {B}}}
    \newcommand{\CC}{{\mathcal {C}}} \renewcommand{\CD}{{\mathcal {D}}}
    \newcommand{\CE}{{\mathcal {E}}} \newcommand{\CF}{{\mathcal {F}}}
    \newcommand{\CG}{{\mathcal {G}}} \newcommand{\CH}{{\mathcal {H}}}
    \newcommand{\CI}{{\mathcal {I}}} \newcommand{\CJ}{{\mathcal {J}}}
    \newcommand{\CK}{{\mathcal {K}}} \newcommand{\CL}{{\mathcal {L}}}
    \newcommand{\CM}{{\mathcal {M}}} \newcommand{\CN}{{\mathcal {N}}}
    \newcommand{\CO}{{\mathcal {O}}} \newcommand{\CP}{{\mathcal {P}}}
    \newcommand{\CQ}{{\mathcal {Q}}} \newcommand{\CR}{{\mathcal {R}}}
    \newcommand{\CS}{{\mathcal {S}}} \newcommand{\CT}{{\mathcal {T}}}
    \newcommand{\CU}{{\mathcal {U}}} \newcommand{\CV}{{\mathcal {V}}}
    \newcommand{\CW}{{\mathcal {W}}} \newcommand{\CX}{{\mathcal {X}}}
    \newcommand{\CY}{{\mathcal {Y}}} \newcommand{\CZ}{{\mathcal {Z}}}

    \newcommand{\RA}{{\mathrm {A}}} \newcommand{\RB}{{\mathrm {B}}}
    \newcommand{\RC}{{\mathrm {C}}} \newcommand{\RD}{{\mathrm {D}}}
    \newcommand{\RE}{{\mathrm {E}}} \newcommand{\RF}{{\mathrm {F}}}
    \newcommand{\RG}{{\mathrm {G}}} \newcommand{\RH}{{\mathrm {H}}}
    \newcommand{\RI}{{\mathrm {I}}} \newcommand{\RJ}{{\mathrm {J}}}
    \newcommand{\RK}{{\mathrm {K}}} \newcommand{\RL}{{\mathrm {L}}}
    \newcommand{\RM}{{\mathrm {M}}} \newcommand{\RN}{{\mathrm {N}}}
    \newcommand{\RO}{{\mathrm {O}}} \newcommand{\RP}{{\mathrm {P}}}
    \newcommand{\RQ}{{\mathrm {Q}}} \newcommand{\RR}{{\mathrm {R}}}
    \newcommand{\RS}{{\mathrm {S}}} \newcommand{\RT}{{\mathrm {T}}}
    \newcommand{\RU}{{\mathrm {U}}} \newcommand{\RV}{{\mathrm {V}}}
    \newcommand{\RW}{{\mathrm {W}}} \newcommand{\RX}{{\mathrm {X}}}
    \newcommand{\RY}{{\mathrm {Y}}} \newcommand{\RZ}{{\mathrm {Z}}}

    \newcommand{\fa}{{\mathfrak{a}}} \newcommand{\fb}{{\mathfrak{b}}}
    \newcommand{\fc}{{\mathfrak{c}}} \newcommand{\fd}{{\mathfrak{d}}}
    \newcommand{\fe}{{\mathfrak{e}}} \newcommand{\ff}{{\mathfrak{f}}}
    \newcommand{\fg}{{\mathfrak{g}}} \newcommand{\fh}{{\mathfrak{h}}}
    \newcommand{\fii}{{\mathfrak{i}}} \newcommand{\fj}{{\mathfrak{j}}}
    \newcommand{\fk}{{\mathfrak{k}}} \newcommand{\fl}{{\mathfrak{l}}}
    \newcommand{\fm}{{\mathfrak{m}}} \newcommand{\fn}{{\mathfrak{n}}}
    \newcommand{\fo}{{\mathfrak{o}}} \newcommand{\fp}{{\mathfrak{p}}}
    \newcommand{\fq}{{\mathfrak{q}}} \newcommand{\fr}{{\mathfrak{r}}}
    \newcommand{\fs}{{\mathfrak{s}}} \newcommand{\ft}{{\mathfrak{t}}}
    \newcommand{\fu}{{\mathfrak{u}}} \newcommand{\fv}{{\mathfrak{v}}}
    \newcommand{\fw}{{\mathfrak{w}}} \newcommand{\fx}{{\mathfrak{x}}}
    \newcommand{\fy}{{\mathfrak{y}}} \newcommand{\fz}{{\mathfrak{z}}}
     \newcommand{\fA}{{\mathfrak{A}}} \newcommand{\fB}{{\mathfrak{B}}}
    \newcommand{\fC}{{\mathfrak{C}}} \newcommand{\fD}{{\mathfrak{D}}}
    \newcommand{\fE}{{\mathfrak{E}}} \newcommand{\fF}{{\mathfrak{F}}}
    \newcommand{\fG}{{\mathfrak{G}}} \newcommand{\fH}{{\mathfrak{H}}}
    \newcommand{\fI}{{\mathfrak{I}}} \newcommand{\fJ}{{\mathfrak{J}}}
    \newcommand{\fK}{{\mathfrak{K}}} \newcommand{\fL}{{\mathfrak{L}}}
    \newcommand{\fM}{{\mathfrak{M}}} \newcommand{\fN}{{\mathfrak{N}}}
    \newcommand{\fO}{{\mathfrak{O}}} \newcommand{\fP}{{\mathfrak{P}}}
    \newcommand{\fQ}{{\mathfrak{Q}}} \newcommand{\fR}{{\mathfrak{R}}}
    \newcommand{\fS}{{\mathfrak{S}}} \newcommand{\fT}{{\mathfrak{T}}}
    \newcommand{\fU}{{\mathfrak{U}}} \newcommand{\fV}{{\mathfrak{V}}}
    \newcommand{\fW}{{\mathfrak{W}}} \newcommand{\fX}{{\mathfrak{X}}}
    \newcommand{\fY}{{\mathfrak{Y}}} \newcommand{\fZ}{{\mathfrak{Z}}}

\title{On a Diophantine inequality involving a prime and an almost-prime}%
\author{Liyang Yang}%

\address{Department of Mathematical Sciences \\
 Tsinghua University \\
 Beijing \\
 100084\\
 P. R. China}

\email{yly12@mails.tsinghua.edu.cn}
\date{\today}%

\begin{abstract}
We prove that there are infinitely many solutions of
$$
|\lambda_0+\lambda_1p+\lambda_2P_r|<p^{-\tau},
$$
where $r=3,$ $\tau=\frac1{118}$, and $\lambda_0$ is an arbitrary real number and $\lambda_1,\lambda_2\in\BR$
with $\lambda_2\neq0$ and $0>\frac{\lambda_1}{\lambda_2}$ not in $\mathbb{Q}$.
This improves a result by Harman. Moreover, we show that one can require the prime $p$ to be of the form $\floor{n^c}$ for some positive integer $n$, i.e. $p$ is a Piatetski-Shapiro prime, with $r=13$ and $\tau=\rho(c),$ a constant explicitly determined by $c$ supported in $\left(1, 1+\frac1{149}\right].$
\end{abstract}
\maketitle

\tableofcontents

\section{Introduction}\label{s:1}

In Diophantine Approximation, a classical theorem of
Kronecker (\cite{2:Gal67}, Theorem 440)
indicates that there are infinitely
many solutions in positive integers $n_1,n_2$ of
$$
|\lambda_0+\lambda_1n_1+\lambda_2n_2|
<3\left(\max\left\{\frac{n_1}{\lambda_2},\frac{n_2}{\lambda_1}
\right\}\right)^{-1},
$$
where $\frac{\lambda_1}{\lambda_2}$ is irrational and $\lambda_0$ is
an arbitrary real number.

The case where $n_1$ and $n_2$ are both primes is of great interest and
remains open to date (\cite{7:Ram77}, \cite{8:Sri82}).
The first approximation in this direction has been given by Vaughan
\cite{Vau76} who
proved that
there are infinitely many solutions of
$$
|\lambda_0+\lambda_1p+\lambda_2P_4|<p^{-1/600000},
$$
where and henceforth in this paper the letter $p$ denotes a prime and
$P_r$ a number with at most $r$ prime factors. Harman \cite{1:Har84} proved that
there are infinitely many solutions of
\begin{equation}\label{eqn:1}
|\lambda_0+\lambda_1p+\lambda_2P_3|<p^{-\tau},
\end{equation}
with $\tau=\frac1{300}$.

In this paper,
we will improve Harman's result by showing that in \eqref{eqn:1} one can
actually take $\tau=\frac1{118}$. One of the main results of
this paper will be the following.
\begin{thm}\label{p:main}
For $\lambda_0,\lambda_1,\lambda_2\in\BR$
with $\frac{\lambda_1}{\lambda_2}$ both negative and irrational,
there are infinitely many solutions of
$$
|\lambda_0+\lambda_1p+\lambda_2P_3|<p^{-\frac1{118}}.
$$
\end{thm}
Moreover, recall that in \cite{HB83} Heath-Brown proved Pjatecki-$\check{S}apiro$ prime number theorem, i.e.
\begin{align*}
\pi_{c}(x):=\sum_{\substack{n\leq x\\\floor{n^c} \text{is a prime}}}1=c^{-1}Li(x)+O\left(x e^{-\delta\sqrt{\log x}}\right) ,
\end{align*}
where $c$ is a real number satisfying that $1<c<\frac{755}{662}=1.1404...$, and $\delta=\delta(c)>0$. Thus we can naturally ask, what will happen if we replace the prime number theorem in the main term by Pjatecki-$\check{S}apiro$ prime number theorem? Can we require the prime $p$ in Theorem \ref{p:main} to be a Pjatecki-$\check{S}apiro$ prime?

The answer is positive, although at cost of increasing the number of factors of the corresponding almost-prime, and we will give a concrete describe about it as follows.
\begin{thm}\label{p:main2}
For $c\in \left(1, 1+\frac1{149}\right]$, $\lambda_0,\lambda_1,\lambda_2\in\BR$
with $\frac{\lambda_1}{\lambda_2}$ both negative and irrational,
there are infinitely many solutions of
$$
\left|\lambda_0+\lambda_1\tilde{p}+\lambda_2P_{13}\right|<\tilde{p}^{-\rho(c)},
$$
where $\tilde{p}$ is a prime of the form $\floor{n^c}$ for some positive integer $n$ and
$$
\rho(c):=\frac{1+9(c^{-1}-1)}{12}-\frac{c}{13-0.144}.
$$
\end{thm}
\begin{remark}
We can take $\rho(c)=\frac1{180}$, when $c=1+2\times10^{-10}$.
\end{remark}

\textbf{Acknowledgements.}
abc

\section{Notation and outline of the method}\label{s:2}
\subsection{Notation}
We shall use $\eta$ and $\varepsilon$ for arbitrary small positive numbers (especially we require $\varepsilon\leq\eta\leq 10^{-12}$);
 and sometimes they may be slightly different in context
 just for simplicity.

We write $\floor{x}$ for the largest integer not exceeding $x$.
We write $\|x\|$
for the distance from $x$ to a nearest integer and $\round{x}$ for the nearest
integer to $x$ when $\|x\|\neq\frac12$.
Clearly we may assume that $\lambda_1>0$ and $\lambda_2=-1$. Let
$\frac {a'}q$ be a
convergence to the continued fraction for $\lambda_1$ and assume $q$ to be
quite large in terms of $\lambda_0$, $\lambda_1$ and $\lambda_1^{-1}$;
let $X$ be a large number
such that $q\asymp X^{\frac13+\rho+\eta}$. Trivially, one
can write $\lambda_0=\frac bq+\gamma$ with $|\gamma|<\frac 1q$.

As in \cite{1:Har84}, we assume that $q$ is so large that
$\min\{\frac {a'}q,\frac q{a'}\}>X^{-\frac\rho4}$
and $a'X+b'<qX^{1+\frac\eta4}.$
In this paper, $p$, $\tilde{p}$, $p_i, i=1,2,\dots$ represent primes; $\sum^\flat$ indicates that the summation is only over square-free numbers.
For convenience, we shall denote by
\begin{align*}
e(x)&:=\exp(2\pi ix),\qquad\xi:=X^{-\rho},\quad\text{where $\rho$ is a positive number}; \displaybreak[0]\\
P(z)&:=\prod_{p\leq z}p,\qquad Y:=\floor{3\xi^{-1}X^\eta};\displaybreak[0]\\
\pi_{c}(x)&:=\sum_{\substack{n\leq x\\\text{$\floor{n^c}$ is a prime}}}1,\qquad \pi(x):=\sum_{\substack{p\leq x\\\text{$p$ is a prime}}}1.
\end{align*}
\subsection{The weighted sieve}
Essentially, to prove Theorem \ref{p:main}, if we use the same method as in \cite{1:Har84} but with
 a parameterized weight to optimize the result, we
will obtain that $\tau=\frac{1}{147}$ is admissible as mentioned in Section 6. However, one can expect to obtain a better result by using Buchstab's sifting weights in \cite{Lab79} rather than Richert's weight $w_p:=1-\frac{u\log p}{\log X}$, together with Selberg's trick, as in \cite{IL81}. We will show
in Theorem \ref{p:thm1} that some terms in the resulting sums can be estimated more efficiently
by using a 2-dimensional sieve, rather than using the linear sieve only. The 2-dimensional
 sieve helps us sieve primes in a much larger range, which will give a better result. Moreover, combining with Chen's idea, i.e., the so-called Switching Principle, as in \cite{1:Har84}, we can thus improve Harman's result. The last step is to work out the restrictions of those parameters
both from main terms and error terms explicitly, and then figure
out the optimal results from them, which can be done by
\emph{Mathematica 9}.\\

We will put the proof Theorem \ref{p:main2} in the last section, as it's somewhat similar to that of Theorem \ref{p:main}. For instance, the exponential sums appearing in the error terms can actually be divided into two parts roughly, one of which can actually be handled by results in Section \ref{error:1}. Nevertheless, we need a lemma to estimate the other part because it is an exponential sum of analytic type. All these will be done in Section \ref{s:15}.

Also, we will cover a slight gap of \cite{1:Har84} in Section \ref{error:1}.
\begin{remark}
Selberg's trick can often help us slightly expand the range of sifting, e.g. see \cite{Irv 15}, where the sifting set is naturally multiplicative by the Chinese reminder theorem, and thus is easier to handle. However, the sifting set here has no multiplicative structure, so we have to use other tricks to conquer.
\end{remark}

 As it points out in \cite{1:Har84} it
suffices to show that the number of solutions of
$$
\left|\frac {b'}q+\frac{pa'}q-P_3\right|<\frac{X^{-\rho}}2
$$
tends to infinity with $X$. Here
$p < X$, $P_3<\frac{a'X+b'}q$.
Hence, we will work with the set
\begin{align*}
\CA&:=\left\{\round{\frac{b'+pa'}q}:
p\leqslant X,\left\|\frac{b'+pa'}q\right\|<\frac\xi2
\right\}. \\
\end{align*}

Here we list all notation used in the sieve method:
\begin{align*}
\CH_r&:=\{n\in\CH:r\mid n\},\quad\text{for any finite set of positive integers }\CH;\\
\fN(\beta)&:=\Bigg\{p_1p_2p_3p_4:X^\beta\leqslant p_1<2X^\beta,
p_1\leqslant p_2\leqslant\left(\frac{a'X+b'}{qp_1}\right)^{\frac13}, \\
& p_2\leqslant p_3\leqslant\left(\frac{a'X+b'}{qp_1p_2}\right)^{\frac12},
X^{\frac\alpha4}\leqslant p_4\leqslant\frac{a'X+b'}{qp_1p_2p_3}
\Bigg\}; \displaybreak[0]\\
\CA(\beta)^*&:=\left\{
n:n\leqslant X,\left\|\frac{b'+na'}q\right\|<\frac\xi2,
\round{\frac{b'+na'}q}\in\fN(\beta)
\right\}; \displaybreak[0]\\
\fP_r&:=\{n\in\BN:n\text{ has at most }r\text{ prime divisors}\};
\displaybreak[0]\\
R_d&:=\#\CA_d-\frac{\pi(X)\xi}{d};\qquad
\CS:=\sum_{n\in\CA\cap\fP_3}1;\\
\widetilde{w_p}&:=
\begin{cases}
\displaystyle cw_p, \hspace{0.5cm} if \quad p=P_n \quad or \quad p\geq x^{b/a};\\[0.5em]
\displaystyle \min \left(cw_p, c-b-1+a\frac{\log P_n}{\log x}\right),\hspace{0.5cm}otherwise,
\end{cases}\\
&\quad \text{where $1\leq b \leq c \leq a=cu$ and $w_p:=1-u\frac{ \log p}{\log x}$}.
\displaybreak[0]\\
\CW(\CA,u,\lambda)&:=\sideset{}{^\flat}\sum_{\substack{s\in\CA\\
\left(s,P\left(X^{\frac1a}\right)\right)=1
}}\left(1-\lambda \sum_{\substack{X^{\frac1a}\leqslant p
\leqslant X^{\frac ca}\\
p\mid s}}\widetilde{w_p}\right)
+\sum_{p\geqslant X^{\frac1a}}\sum_{\substack{h\in\CA\\
p^2\mid h}}1; \displaybreak[0]\\
\CS(\CA(\beta)^*,z)&:=\sum_\beta\sum_{\substack{n\in\CA(\beta)^*\\
(n,P(z))=1}}1;
\end{align*}
where
$0<\frac4a\leqslant 4\beta\leqslant 1,$ both are undetermined parameters.

Define
$$
\CJ(\lambda):=\CW(\CA,u,\lambda)-\lambda\CS(\CA(\beta)^*,X^{\frac12-\eta}).
$$

For simplicity, we shall denote by $z:=X^{\frac1a}$,
$y:=X^{\frac ca}$.
\begin{lemma}
Assume that $b=1$ or $b>1$ such that $a\geq3c+b+1$, then we have
\begin{equation}\label{eqn:3}
\CS\geq\CJ(\lambda) \quad\text{if}\quad \lambda^{-1}<5c-a.
\end{equation}
\end{lemma}
\begin{proof}
Notice that
$$
\CS=\sideset{}{^\flat}\sum_{\substack{s\in\CA\cap\fP_3\\
\left(s,P\left(X^{\frac\alpha4}\right)\right)=1
}}1+O\left(X^{1-\frac\alpha4}\right),
$$
thus we only need the following inequality:
\begin{multline}\label{t:1}
\sideset{}{^\flat}\sum_{\substack{s\in\CA\setminus\fP_3\\
\left(s,P\left(X^{\frac\alpha4}\right)\right)=1
}}1\leqslant
\lambda\sideset{}{^\flat}\sum_{\substack{s\in\CA\\
\left(s,P\left(X^{\frac\alpha4}\right)\right)=1
}}\sum_{X^{\frac\alpha4}\leqslant p\leqslant X^{\frac1u}}
\widetilde{w_p}\\
+\lambda\sum_\beta\sum_{\substack{n\in\CA(\beta)^*\\
\left(n,P\left(X^{\frac12-\eta}\right)\right)=1}}1
+O\left(X^{1-\frac\alpha4}\right),
\end{multline}
with the assumption that $0<\rho<\frac\alpha4$.
To this end, we divide it into two cases:

\begin{description}
\item[Case 1]
$s\in\CA\setminus\fP_4$, so that $s$
has at least $5$ prime factors.
If $s$ has a prime factor $p$ which is larger than $P_s$ and
$$
\frac{\log p}{\log X}\leq \frac{b+1}{a}-\frac{\log P_s}{\log X},
$$
then
$$
\sum_{p\mid s}\widetilde{w_p}\geq c-a\frac{\log P_s}{\log X}+c-b-1+a\frac{\log P_s}{\log X}=2c-b-1\geq 5c-a.
$$
Otherwise, every prime divisor of $s$ which is larger than $P_s$ must satisfy
$$
\frac{\log p}{\log X}\geq \frac{b+1}{a}-\frac{\log P_s}{\log X},
$$
which means that
$$
\widetilde{w_p}\geq c-a\frac{\log p}{\log X}, \qquad \text{for all}\quad p\mid s.
$$
This provides that
$$
\sum_{p\mid n}\widetilde{w_p}\geq c\omega(s)-a\frac{\log s}{\log X}\geq 5c-a.
$$
thus we have (\ref{t:1}) because of
$$
\lambda^{-1}<5c-a.
$$
\item[Case 2] $s\in\CA\cap\fP_4$ Similarly as above, we have
$\sum_{p\mid n}\widetilde{w_p}\geq 4c-a.$
So (\ref{t:1}) comes from the assumption that
$\lambda^{-1}<5c-a.$
\end{description}
\end{proof}

Therefore, we have
\begin{cor}
For $\lambda^{-1}<5c-a$, if
$$
\CJ(\lambda):=\CW(\CA,u,\lambda)-\lambda\CS(\CA(\beta)^*,X^{\frac12-\eta})
\gg\frac{\pi(X)\xi}{\log X},
$$
then theorem \ref{p:main} holds with $\tau=\rho.$
\end{cor}

In the following sections, we will prove that
$\CJ(\lambda)\gg\frac{\pi(X)\xi}{\log X}$ and we can take
$\rho=\frac1{118}$.
\section{Some auxiliary lemmas}
\begin{lemma}
For any $x\geq2$, we have
$$
\prod_{p\leq x}\left(1-\frac1p\right)=\frac {e^{-\gamma}}{\log x}\left(1+O\left(\frac1{\log x}\right)\right);
$$
$$
\sum_{p\leq x}\frac1p=\log\log x+c+O\left(\frac1{\log x}\right),
$$
where $c$ is an absolute constant.
\end{lemma}
\begin{remark}
These two estimates are usually called Mertens formulas.
\end{remark}
\begin{lemma}[\cite{Mont78}]\label{5}
Let $\delta_0<\frac {1}{2}$ and $\chi(t)$ be the characteristic function of interval $(-\delta_0,\delta_0)$ extended to be periodic with period 1,then there exists $A(t)$, $B(t)$ such that
$$
A(t)\leq\chi(t)\leq B(t)£¬
$$
where $A(t)$, $B(t)$ can be written as
\begin{align*}
A(t)&:=2\delta_0-(N+1)^{-1}+\sum_{1\leq|n|\leq N}A_ne(nt), \\
B(t)&:=2\delta_0+(N+1)^{-1}+\sum_{1\leq|n|\leq N}B_ne(nt),
\end{align*}
with coefficients $A_n, B_n$ satisfying $\max\{|A_n|,|B_n|\}\ll\delta_0$, for $1\leq|n|\leq N$.
\end{lemma}

\begin{lemma}\label{p:lem1-new}
Suppose that $0\leq \alpha<\beta\leq1$ and $\Delta>0$ with $2\Delta<\beta-\alpha$, then there exists a smooth function $\chi$ with the period 1 satisfying that:
\begin{description}
\item[(1)]
$\chi(x)=1$ if $\alpha+\Delta\leq\{x\}\leq\beta-\Delta$, $\chi(x)=0$ if $\{x\}\leq\alpha$ or $\{x\}\geq\beta$, and $\chi(x)\in\left[0,1\right]$ otherwise.
\item[(2)]
$\chi(x)=\beta-\alpha+\sum_{1\leq|h|\leq\Delta^{-1-\varepsilon}}{c_h e(hx)}+O\left(\Delta\right)$, where
$$
c_h\ll_\varepsilon \min\{\frac1{|h|},\beta-\alpha-\Delta\}.
$$
Moreover, the function $g(x):=\sum_{1\leq|h|\leq\Delta^{-1-\varepsilon}}{c_h e(hx)}$ is real.
\end{description}
\end{lemma}
\begin{proof}
Fix $\varepsilon>0$ small enough. Then by (\cite{V54} Lemma 12, Chapter 1) we have
$$\chi(x)=\beta-\alpha-\Delta+\sum_{|h|\geq1}{\left(a_j\cos{2\pi jx}+b_j\sin{2\pi jx}\right)}.$$
 Take $c_j=\frac{a_j-ib_j}{2}$ and $c_{-j}=\frac{a_j+ib_j}{2}$ for any $j\in \BN_{\geq 1},$ then by estimations from (\cite{V54}) on $a_j$ and $b_j$, we have
$$
c_h\ll_\varepsilon \min\{\frac1{|h|},\beta-\alpha-\Delta, \frac1{{\Delta^r}|h|^{r+1}}\} \hspace{0.5cm}\text{for any arbitrary integer $\emph{r}$}.
$$
Take $r$ large enough such that $\frac1r\leq\varepsilon$ and $H:=\Delta^{-\frac{r+1}{r}}$, then
$$
\sum_{|h|\geq H}c_h e(hx)\ll \sum_{|h|\geq H}\frac{1}{{\Delta^r}|h|^{r+1}}\ll \frac{1}{{\Delta^r}|H|^{r}}\ll \Delta.
$$
Obviously, $g(x)=\sum_{1\leq|h|\leq\Delta^{-1-\varepsilon}}{\left(a_j\cos{2\pi jx}+b_j\sin{2\pi jx}\right)}$ is a real function.
\end{proof}
Set $$\CS_{\widetilde{w}}(\CA):=\sum_{\substack{s\in\CA\\
\left(s,P\left(X^{\frac\alpha4}\right)\right)=1
}}\sum_{\substack{X^{\frac\alpha4}\leqslant p
\leqslant X^{\frac1u}\\
p\mid s}}\widetilde{w_p},$$
 then by a direct computation we have
\begin{lemma}\label{pp}
\begin{align*}
\CS_{\widetilde{w}}(\CA)&=(1-\frac bc)\sum_{X^{\frac 1a}\leq p<X^{\frac ba}}\CS\left(\CA_p, X^{\frac1a}\right)
+u\int_{\frac1a}^{\frac{b+1}{2a}}\left(\sum_{X^{s}\leq p\leq X^{\frac{b+1}{a}-s}}\CS\left(\CA_p, X^{s}\right)ds\right)\\
&\qquad+\sum_{X^{\frac 1a}\leq p<X^{\frac {b+1}{2a}}}\left(\frac{b+1}{c}-\frac{2u\log p}{\log X}\right)\CS\left(\CA_p, p\right)
+\sum_{X^{\frac ba}\leq p<X^{\frac 1u}}w_p\CS\left(\CA_p, X^{\frac1a}\right).\\
\end{align*}
\end{lemma}
\section{Estimates for exponential sums I}\label{error:1}
Our main goal in this section is to prove that
\begin{equation}\label{4}
\sum_{d\leqslant X^\alpha}\frac\xi d\max_{N\leqslant X}
\sum_{l=1}^{dY}\left|\sum_{n\leqslant N}\Lambda(n)
e\left(\frac{anl}{dq}\right)\right|\ll{\xi\pi(X)}X^{-\eta}
\end{equation}
with $\alpha$ as large as possible.

However, the lemmas in \cite{1:Har84} can only give the result without
 taking $max$ between the two sums. We should point out that with
 some slight modifications of the proof in \cite{1:Har84} we will be able to prove \eqref{4}.

This is a generalization of \cite{1:Har84}, Lemma 3:

\begin{lemma}\label{p:lem9-new}
Suppose $X,M\geqslant 1$, $\delta>0$, $\CM$ a set of $\leqslant T$
integer points $(l,m)$ with $M\leqslant m<2M$, $\lambda_{lm}$ real numbers
for $(l,m)\in\CM$, and $\{a_n\}$ a sequence of complex numbers, then
$$
\sum_{(l,m)\in\CM}\max_{N\leqslant X}\left|
\sum_{mn\leqslant N}a_ne(\lambda_{lm}n)\right|^2\ll
D_\delta\log^3(2TX)\left(\frac XM+\delta^{-1}\right)\sum_{n\leqslant X/M}|a_n|^2,
$$
where
$$
D_\delta=\max_{(l,m)\in\CM}\#\left\{(l',m')\in\CM:
\|\lambda_{lm}-\lambda_{l'm'}\|<\delta\right\}.
$$
\end{lemma}
\begin{proof}
Define
$$
\delta(\beta):=\begin{cases}
1,&\text{if }0\leqslant\beta\leqslant\gamma, \\
0,&\text{otherwise},
\end{cases}
$$
which is a truncation function. Then we have
$$
\delta(\beta)=\int_{-A}^Ae^{i\beta t}\frac{\sin\gamma t}{\pi t}\mathrm dt
+O\left(\frac1{A|\gamma-\beta|}\right)
$$
as in the proof of Lemma 2 of \cite{Vau80}.
Here we take $A=2TX$, $\gamma_{lm}=\log\left(N_{lm}+\frac12\right)$, for
$(l,m)\in\CM.$
where
p$$
N_{lm}=\max\left\{n_0\in X:\max_{N\leqslant X}
\left|\sum_{mn\leqslant N}a_ne(\lambda_{lm}n)\right|^2
=\left|\sum_{mn\leqslant n_0}a_ne(\lambda_{lm}n)\right|^2
\right\}.
$$

Then we have
\begin{align*}
&\quad\sum_{(l,m)\in\CM}\max_{N\leqslant X}\left|\sum_{mn\leqslant N}
a_ne(\lambda_{lm}n)\right|^2
=\sum_{(l,m)\in\CM}\left|\sum_{mn\leqslant N_{lm}}
a_ne(\lambda_{lm}n)\right|^2
\displaybreak[0]\\
&\leq\sum_{(l,m)\in\CM}\left(\int_{-A}^A\left|\sum_n
a_nn^{it}e(\lambda_{lm}n)\right|\cdot\left|
\frac{\sin\gamma_{lm} t}{\pi t}\right|\mathrm dt\right)^2 \\
&\qquad{}+O\left(\sum_{(l,m)\in\CM}\left(\sum_n|a_n|\frac1
{A\log\frac{N_{lm}+1/2}{mn}}
\right)^2\right)
\displaybreak[0]\\
&\ll\sum_{(l,m)\in\CM}\int_{-A}^A\left|\sum_n
a_nn^{it}e(\lambda_{lm}n)\right|^2\cdot\min\{{\gamma_{lm},\frac1{|t|}}\} \mathrm dt\cdot\log A \\
&\qquad{}+O\left(\sum_{(l,m)\in\CM}\left(\sum_n|a_n|\frac1
{A\log\frac{N_{lm}+1/2}{N_{lm}}}
\right)^2\right)
\displaybreak[0]\\
&\ll\log A\cdot\int_{-A}^A\left(\sum_{(l,m)\in\CM}\left|\sum_n
a_nn^{it}e(\lambda_{lm}n)\right|^2\right)\cdot \min\{{\log X,\frac1{| t|}}\}\mathrm dt \\
&\qquad{}+O\left(\sum_{(l,m)\in\CM}\left(\sum_n|a_n|\frac1
{A\log\frac{X+1/2}{X}}
\right)^2\right)
\displaybreak[0]\\
&\ll\log A\cdot\left(\sum_{(l,m)\in\CM}\left|\sum_n
a_nn^{it}e(\lambda_{lm}n)\right|^2\right)\cdot \int_{-A}^A \min\{{\log X,\frac1{|t|}}\}\mathrm dt \\
&\qquad{}+O\left(\sum_{(l,m)\in\CM}\left(\sum_n|a_n|\frac1
{A\log\frac{X+1/2}{X}}
\right)^2\right)
\displaybreak[0]\\
&\ll D_\delta\log^2(2TX)\left(\frac XM+\delta^{-1}\right)\left(\int_0^1
\log X\mathrm dt+\int_1^A \frac 1t\mathrm dt\right)\sum_n|a_n|^2
\displaybreak[0]\\
&\ll
D_\delta\log^3(2TX)\left(\frac XM+\delta^{-1}\right)\sum_n|a_n|^2,
\end{align*}
where the last step comes from \cite{1:Har84}, Lemma 3.
\end{proof}

This is a generalization of \cite{1:Har84}, Lemma 5:

\begin{lemma}\label{p:lem10-new}
Suppose $\varepsilon>0$, $X>R$, $J,M\geqslant 1$, $1<q\leqslant X^{\frac34}$,
$\log|a|\ll\log X$, $(a,q)=1$, then
$$
\sum_{r\sim R}\max_{N\leqslant X}\sum_{j\sim J}\sum_{m\sim M}
\left|\sum_{mn\leqslant N}e\left(\frac{ajmn}{rq}\right)\right|
\ll X^\varepsilon\left(\frac{JX}q+RJM+qR^2\right).
$$
\end{lemma}
\begin{proof}
By lemma 3 of \cite{Vau77} we obtain
\begin{multline*}
\sum_{r\sim R}\max_{N\leqslant X}\sum_{j\sim J}\sum_{m\sim M}
\left|\sum_{mn\leqslant N}e\left(\frac{ajmn}{rq}\right)\right| \\
\ll\log X(JM)^{\frac\varepsilon3}\sum_{r\sim R}\left(
\frac{JX\cdot(r,a)}{rq}+JM+qR
\right).
\end{multline*}

Hence, it follows from the same estimates in lemma 5 of \cite{1:Har84}.
\end{proof}

This is a generalization of \cite{1:Har84}, Lemma 7:

\begin{lemma}\label{p:lem11-new}
Suppose that $\varepsilon>0$, $X\geqslant R$, $L,M\geqslant 1$,
$1<q\leqslant X$, $(a,q)=1$ and $a\asymp q$,
$\max\left\{\frac{LM}{qR},\frac{qM}X\right\}<1$,
$a_n,b_m\ll X^\varepsilon$. Then
\begin{multline*}
\sum_{r\sim R}\max_{N\leqslant X}\sum_{l\sim L}
\left|\sum_{m\sim M}b_m\sum_{mn\leqslant N}
a_ne\left(\frac{lmna}{qr}\right)\right| \\
\ll X^{1+3\varepsilon}R\left(L+\frac RM\right)
\left(\frac MX+\frac1{MRL+R^2}\right)^{\frac12}.
\end{multline*}
\end{lemma}
\begin{proof}
The proof is essentially the same as that of lemma 7 of \cite{1:Har84},
with lemma 3 of \cite{1:Har84} replaced by lemma \ref{p:lem9-new} above.
\end{proof}

This is a generalization of \cite{1:Har84}, Lemma 8:

\begin{lemma}\label{p:lem6}
Suppose that $X,R,L\geqslant 1$, $a\asymp q$, $(a,q) = 1$, $\varepsilon > 0$ and
$\frac{TX^{\frac13}}R<q<X^{\frac23}$, where $T=\max\{L,R\}$.
Then we have
$$
\sum_{r\sim R}\max_{N\leqslant X}
\sum_{l\sim L}\left|\sum_{n\leqslant N}\Lambda(n)
e\left(\frac{anl}{rq}\right)\right|
\ll X^\varepsilon\left(X^{\frac23}TR+X^{\frac{11}{12}}(TR)^{\frac12}\right).
$$
\end{lemma}
\begin{proof}
Using Vaughan's identity we split the inner sum above into
$\ll\log N$ sums of the form
$$
\sum_{m\sim M}\sum_{mn\leqslant N}a_nb_me\left(\frac{nalm}{dq}\right),
$$
with either

(I) $a_n=1$ or $\log n$, $M<X^{\frac23}$, $b_m\ll X^\varepsilon$, or

(II) $a_n,b_m\ll X^\varepsilon$, $X^{\frac16}<M<X^{\frac13}$.

Sums of type (I) can be handled by lemma \ref{p:lem10-new} and sums
of type (II) by lemma \ref{p:lem11-new} and the estimate above follows.
\end{proof}

\begin{cor}\label{p:cor13-new}
We have
$$
\sum_{d\leqslant X^\alpha}\frac\xi d\max_{N\leqslant X}
\sum_{l=1}^{dY}\left|\sum_{n\leqslant N}\Lambda(n)
e\left(\frac{anl}{dq}\right)\right|\ll{\xi\pi(X)}X^{-\eta}.
$$
\end{cor}

\section{Sieve estimates}\label{s:3}

Let $f_1$, $F_1$ and $F_2$ be the limit functions occurred in Beta-Sieve,
which are given by the following definition:
\begin{align*}
f_1(s)&:=A_1s^{-1}\log(s-1)\qquad\text{ for }2\leqslant s\leqslant 4; \\
f_1(s)&:=A_1s^{-1}\left(\log(s-1)+\int_{3}^{s-1}\frac{du}{u}\int_{2}^{u-1}\frac{\log(v-1)}{v}dv\right)\qquad\text{ for }4\leqslant s\leqslant 6;\\
F_1(s)&:=A_1s^{-1}\qquad\text{ for }s\leqslant 3; \\
F_1(s)&:=A_1s^{-1}\left(1+\int_{2}^{s-1}\frac{\log(v-1)}{v}dv\right)\qquad\text{ for }3\leqslant s\leqslant 5;\\
F_2(s)&:=A_2s^{-2}\qquad\text{ for }s\leqslant\beta_2+1,
\end{align*}
where $A_1=2e^{\gamma}$, $\beta_2=4.8333\cdots$,
$A_2=43.496\cdots$ are defined in \cite{FI10}, Chapter 11.
We can, with a patient calculation, show that for $s\in[\beta_2+1, \beta_2+2)$, we have
\begin{align*}
F_2(s)&=s^{-2}\left(\frac{2A_2\log{\beta_2}}{s-1}+C_0+2A_2\log^2(s-1)+4A_2\log(s-1)\right)\\
&\qquad-\frac{4A_2\left(1+s\log (s-1)\right)}{s^2(s-1)},
\end{align*}
where $C_0$ is determined by $F_2(\beta_2+1)=\frac{A_2}{(\beta_2+1)^2}.$
As shown in Lemma \ref{p:lem4-new} below, the level of distribution of $\CA$ can be taken as $\theta_1=\frac13-\rho-\varepsilon.$ Henceforth, we take $a=\frac{\vartheta}{\theta_1}$ and optimize $\vartheta$ to get a better upper bound of $\rho.$ Take $z=X^{\frac1a}$ and $y=X^{\frac ca}$ from now on.

\begin{remark}
The limit functions $f_1$ and $F_2$ are actually defined by systems of differential equations piecewise respectively. $f_1$ is increasing rapidly and very close to its limit 1 when $s\geq 6$. While $F_2$ is decreasing with limit 1. We should point out that in our situation, it turns out that $6\theta-c>\beta_2+1$ since we require that $b\geq 3$, which leads $c$ to be relatively small. Thus the above expression of $F_2$ is invalid. We will discuss this matter in the next section.
\end{remark}

Denote by $A_3:=\frac{A_2}{2e^{2\gamma}}\approx 6.85577$, which will be used in the following section.\\

In this section, we will prove the following theorem,
which improves \cite{1:Har84}, Lemma 1:

\begin{thm}\label{p:thm1}
Let notations be defined as before and assume that $b=1$ or $b>1$ such that $a\geq3c+b+1$, then for any $\delta\in\left[\frac{b}{\vartheta},\frac{c}{\vartheta}\right]$ we have,
\begin{align*}
\mathcal{J}(\lambda)\geq \frac{ae^{-\gamma}(1+o(1))\lambda\xi\pi(X)}{\log X}\mathcal{H}_{\delta}(\vartheta,b,c),
\end{align*}
where
\begin{equation}\label{x:10}
\mathcal{H}_{\delta}(\vartheta,b,c)=2e^{\gamma}\left(A_{\delta}(\vartheta)b+B_{\delta}(\vartheta)c+D_{\delta}(\vartheta)+\mathfrak{F}_{\delta}(\vartheta,c)\right),
\end{equation}
with
\begin{align*}
A_{\delta}(\vartheta)&=-e^{-\gamma}f_1(\vartheta)+\frac1{2e^{\gamma}}\int_{\frac1{\vartheta}}^{\delta}
F_1\left(\vartheta(1-s)\right)\frac{ds}{s}+\frac{1}{\vartheta}\log\frac{1-\delta}{\delta};\\
B_{\delta}(\vartheta)&=e^{-\gamma}f_1(\vartheta)-\frac1{2e^{\gamma}}\int_{\frac1{\vartheta}}^{\delta}
F_1\left(\vartheta(1-s)\right)\frac{ds}{s}
-\frac{2}{a}\mathfrak{I}(\rho);\\
D_{\delta}(\vartheta)&=\frac1{2e^{\gamma}}\mathcal{H}(\vartheta,\vartheta\theta,\vartheta\theta)-\delta\log\frac{1-\delta}{\delta}
+\frac{2\vartheta\theta}{a}\mathfrak{I}(\rho);\\
\mathfrak{F}_{\delta}(\vartheta,c)&=- ae^{-\gamma}\left(\int_{\delta}^{\frac c\vartheta}\left(\frac cs-\vartheta\right)
F_2\left(a\theta_2-\vartheta s\right)\mathrm ds\right).
\end{align*}
\end{thm}

To this end, we need the following lemmas.
\begin{lemma}\label{p:lem4-new}
We have
$$
\CS(\CA,z)\geqslant \xi\pi(X)V(z)\big(f_1(6)+o(1)\big),\\
$$
where $V(z)=\frac{e^{-\gamma}}{\log z}\big(1+o(1)\big)$,
and $z:=X^{\frac1a}$ as mentioned before.
\end{lemma}
\begin{proof}
Take $M\asymp\frac{dX^{\eta}}{\xi}$ in Lemma \ref{5} then we have
\begin{align*}
\#\CA_d&=\sum_{\substack{p\leqslant X\\
d\mid\round{\frac{ap+b}q}\\
\left\|\frac{ap+b}q\right\|<\frac\xi2}}1
=\sum_{\substack{p\leqslant X\\
\left\|\frac{ap+b}{dq}\right\|<\frac\xi{2d}}}1
=\sum_{p\leqslant X}\chi\left(\frac{ap+b}{dq}\right)
\displaybreak[0]\\
&=\frac{\pi(x)\xi}d+E(\CA_d)+O\left(\frac{{\xi\pi(X)}X^{-\eta}}d\right),\\
\end{align*}
where
\begin{align*}
\sum_{p\leqslant X}\sum_{1\leq|l|\leq M}
a_le\left(\frac{(ap+b)l}{qd}\right)\leq E(\CA_d)\leq\sum_{p\leqslant X}\sum_{1\leq|l|\leq M}
b_le\left(\frac{(ap+b)l}{qd}\right)\\
\end{align*}
with $
|a_l|+|b_l|\ll\frac\xi d,\quad \forall 1\leq|l|\leq M.
$

Therefore, by partial summation we have\begin{align*}
E(\CA_d)&\ll\max_{N\leqslant X}\frac1 {\log X}\sum_{1\leq|l|\leq M}\left(|a_l|+|b_l|\right)\left|\sum_{n\leqslant N}\Lambda(n)e\left(\frac{anl}{qd}\right)
\right|\\
&\ll\max_{N\leqslant X}\frac\xi d\sum_{l=1}^{dY}\left|
\sum_{n\leqslant N}\Lambda(n)e\left(\frac{anl}{qd}\right)
\right|.\\
\end{align*}
Hence the density function of sequence $\CA$ is $g_1(d)=\frac1d$;
and thus, by Jurkat-Richert's theorem, we obtain
\begin{align*}
\CS(\CA,z)&\geqslant \xi\pi(X)V(z)\left(f_1(4)+O\left(\left(\log X\right)^{-\frac 1{6}}\right)\right)\\
&\qquad+O\left(\xi\pi(X)X^{-\eta}+
\sum_{d\leqslant X^\alpha}\frac\xi d\max_{N\leqslant X}
\sum_{i=1}^{dY}\left|\sum_{n\leqslant N}\Lambda(n)e\left(\frac{anl}{dq}\right)
\right|\right).
\end{align*}

Then this lemma comes from corollary \ref{p:cor13-new} since $f_1(6)>0.$
\end{proof}

\begin{lemma}\label{p:lem5-new}
If $0<\frac1a<\delta'<\frac ca\leq \theta_1$, let $w=X^{\delta'}$, then
$$
\sum_{z\leqslant p<w}w_p\CS(\CA_p,z)\leqslant{\xi\pi(X)}V(z)\left(
\int_{\frac1a}^{\delta'}\left(\frac1s-u\right)
F_1\left(\frac{4(\alpha-s)}\alpha\right)\mathrm ds+o(1)\right).
$$
\end{lemma}
\begin{proof}Corollary \ref{p:cor13-new} shows that the level of distribution
of $\CA$ is $X^{\theta_1}$. Hence
by Jurkat-Richert's theorem, we have
$$
\CS(\CA_p,z)\leqslant\frac{{\xi\pi(X)}V(z)}p\left(F_1(s_p)
+O\left((\log\frac{X^\alpha}p)^{-\frac16}\right)\right)
+O\left(\sum_{d\leqslant X^{\theta_1}/p}|R_{pd}|\right),
$$
where
$$
s_p=\frac{\log\frac{X^{\theta_1}}{p}}{\log z},\text{ and }
R_{pd}\ll\frac\xi{pd}\max_{N\leqslant X}
\sum_{l=1}^{pdY}\left|\sum_{n\leqslant N}\Lambda(n)e\left(\frac{anl}{pdq}\right)\right|.
$$
Since
\begin{align*}
\sum_{z\leqslant p<w}w_p\sum_{d\leqslant X^{\theta_1}/p}|R_{pd}|
&\ll\sum_{z\leqslant p<w}\sum_{d\leqslant X^{\theta_1}/p}|R_{pd}|\\
&\ll\sum_{d\leqslant X^{\theta_1}}|R_{d}|\sum_{\substack{p\mid d\\z\leqslant p<w}}1\\
&\ll X^{\frac\eta2}\sum_{d\leqslant X^{\theta_1}}|R_{d}| \\
&\ll X^{\frac\eta2}\sum_{d\leqslant X^{\theta_1}}\frac\xi d
\max_{N\leqslant X}\sum_{l=1}^{dY}
\left|\sum_{n\leqslant N}\Lambda(n)
e\left(\frac{anl}{dq}\right)\right|
\ll X^{1-\frac\eta2}\xi,
\end{align*}
we obtain
\begin{align*}
\sum_{z\leqslant p<w}w_p\CS(\CA_p,z)&\leqslant{\xi\pi(X)}V(z)\left(\sum_{z\leqslant p<w}\frac{w_p}pF_1(s_p)+o(1)\right).
\end{align*}

By Mertens formula we have
\begin{align*}
\sum_{t'\leqslant p<t}\frac 1p(1-\frac{\log p}{\log y})
=\log\frac{\log t}{\log t'}-\frac{\log t-\log t'}{\log y}+R(t',t),
\displaybreak[0]
\end{align*}
where
\begin{align*}
R(t',t)&\ll\frac1{t'}, \qquad\text{for any  }t'<t.
\end{align*}

Notice that $F_1$ is bounded and decreasing, so we obtain that
\begin{align*}
\sum_{z\leqslant p<w}\frac{w_p}pF_1(s_p)
&=\int_z^wF_1\left(\frac{\log X^{\theta_1}/t}{\log z}\right)
\mathrm d\sum_{z\leqslant p<t}\frac 1p(1-\frac{\log p}{\log y})\\
&=\int_z^wF_1\left(\frac{\log X^{\theta_1}/t}{\log z}\right)\mathrm d\left(\log\frac{\log t}{\log z}-\frac{\log t-\log z}{\log y}+R(z,t)\right)
\displaybreak[0]\\
&=\int_z^w\frac1t\left(\frac1{\log t}-\frac1{\log y}\right)F_1\left(\frac{\log X^\alpha/t}{\log z}\right)
\mathrm dt+ R(z,t)F_1\left(\frac{\log X^{\theta_1}/t}{\log z}\right)\bigg|_z^w\\
&\qquad{}+O\left(\int_z^w\frac1z\mathrm dF_1\left(\frac{\log X^{\theta_1}/t}{\log z}\right))\right)
\displaybreak[0]\\
&=\int_{\frac1a}^{\delta'}\left(\frac1s-u\right)
F_1\left(a(\theta_1-s)\right)\mathrm ds+O(\frac1z).
\end{align*}

Therefore, we have
$$
\sum_{z\leqslant p<w}w_p\CS(\CA_p,z)\leqslant{\xi\pi(X)}V(z)\left(
\int_{\frac1a}^{\delta'}\left(\frac1s-u\right)
F_1\left(a(\theta_1-s)\right)\mathrm ds+o(1)\right).
$$

This completes the proof.
\end{proof}
Define
$$
\widetilde\CA:=\left\{n\round{\frac{an+b}q}:n\in[z,X],
p\mid\round{\frac{an+b}q},\left\|\frac{an+b}q\right\|<\frac12\xi\right\},
$$
then we have the following auxiliary lemma.
\begin{lemma}
For $d\mid P(z)$ and $p\geq z$, we have
\begin{align*}
\#\widetilde\CA_d
=\frac{X\xi}pg_2(d)+E(X;p,d),
\end{align*}
where
$$
g_2(d):=\prod_{p\mid d}\left(\frac2{p}-\frac1{p^2}\right),
$$
$$
E(X;p,d)\ll q\xi\tau(d)+\frac X{pdq}\sum_{d=d_1d_2}
(q,d_1)(a,pd_2).
$$
\end{lemma}
\begin{proof}
Define $$J:=\{j:|j+b|\leqslant\frac12q\xi\}.$$
\begin{align*}
\#\widetilde\CA_d
&=\sum_{|j|\leqslant\frac12q\xi}\sum_{\substack{n\in[z,X]\\
an+b\equiv j\pmod{pq}\\
n(an+b)\equiv jn\pmod{dq}}}1
=\sum_{j\in J}\sum_{\substack{n\in[z,X]\\
an\equiv j\pmod{pq}\\
an^2\equiv jn\pmod{dq}}}1
\displaybreak[0]\\
&=\sum_{j\in J}\sum_{\substack{n\in[z,X]\\
an-j\equiv 0\pmod{pq}\\
n(an-j)\equiv 0\pmod{dq}}}1
=\sum_{j\in J}\sum_{d=d_1d_2}\sum_{\substack{n\in[z,X]\\
an-j\equiv 0\pmod{pq}\\
n(an-j)\equiv 0\pmod{dq}\\
(n,d)=d_1}}1
\displaybreak[0]\\
&=\sum_{d=d_1d_2}\sum_{j\in J}\sum_{\substack{n\in[z/d_1,X/d_1]\\
ad_1n-j\equiv 0\pmod{pq}\\
n(ad_1n-j)\equiv 0\pmod{d_2q}\\
(n,d_2)=1}}1
=\sum_{d=d_1d_2}\sum_{j\in J}\sum_{\substack{n\in[z/d_1,X/d_1]\\
ad_1n-j\equiv 0\pmod{pd_2q}\\
(n,d_2)=1}}1
\displaybreak[0]\\
&=\sum_{d=d_1d_2}\sum_{\substack{j\in J\\
(ad_1,pd_2q)\mid j}}\left(\frac{\varphi(d_2)}{d_2}\cdot
\frac{X-z}{pdq}(ad_1,pd_2q)+O(1)\right)
\displaybreak[0]\\
&=\sum_{d=d_1d_2}\left(\frac{q\xi}{(ad_1,pd_2q)}+O(1)\right)
\left(\frac{\varphi(d_2)}{d_2}\cdot
\frac{X-z}{pdq}(ad_1,pd_2q)+O(1)\right)
\displaybreak[0]\\
&=\sum_{d=d_1d_2}\frac{\varphi(d_2)}{d_2}\cdot\frac{(X-z)\xi}{pd}
+O\left(\sum_{d=d_1d_2}\frac{q\xi}{(ad_1,pd_2q)}
+\sum_{d=d_1d_2}\frac{X\cdot(ad_1,pd_2q)}{pdq}\right),
\displaybreak[0]\\
\end{align*}
and thus lemma follows by noting that
 $(a,pd_2q)(d_1,pd_2q)\leq(q,d_1)(a,pd_2).$
\end{proof}
Hence $\widetilde\CA$ has a density function $g_2(d)$ with
\begin{align*}
V_2(z)&:=\prod_{p\leqslant z}
\big(1-g_2(p)\big)
=\prod_{p\leqslant z}
\left(1-\frac2{p}+\frac1{p^2}\right) \\
&=\frac{e^{-2\gamma}}{\log^2z}\big(1+o(1)\big)
\qquad\text{by Mertens estimate}.
\end{align*}

We will use Beta-Sieve theory to $\widetilde\CA$ to obtain an upper bound with a larger exponent of level of distribution. To this end, we shall compute its dimension as follows:
\begin{align*}
\sum_{p\leq v}g_2(p)\log p=2\sum_{p\leq v}\big(\frac{\log p}{p}-\frac{\log p}{2p^2}\big)=2\log v+O(1),
\qquad\text{for any }v\geq2.
\end{align*}

Therefore, the sieve dimension is 2. Denote by $\theta_2$ the exponent of level of distribution of $\widetilde\CA$.
\begin{lemma}\label{p:lem6-new}
Assuming $w\leqslant p\leqslant y$ and $p$ is a prime number,
where $w=X^{\delta'}$, $\frac1a\leqslant\delta\leqslant\frac ca\leq\theta_1$,
then we have
$$
\CS(\CA_p,z)\leqslant\frac{X\xi}pV_2(z)
\left(F_2(s_p')+O\left((\log X)^{-\frac16}\right)\right)
+\sum_{d\leq\frac{X^{\theta_2}}{p}}E(X;p,d),
$$
where
$$
s_p':=\frac{\log\left(X^{\theta_2}/p\right)}{\log z}.
$$

\end{lemma}
\begin{proof}
We have
\begin{align*}
\CS(\CA_p,z)&
=\sum_{\substack{n\in\CA_p\\
\left(n,P(z)\right)=1}}1
\displaybreak[0]\\
&=\#\{{p':z\leqslant p'\leqslant X,
p\mid\round{\frac{ap'+b}q},
\left\|\frac{ap'+b}q\right\|<\frac12\xi,
\left(p'\round{\frac{ap'+b}q},P(z)\right)=1}\}\\
&\qquad+\#\{{p':p'<z,
p\mid\round{\frac{ap'+b}q},
\left\|\frac{ap'+b}q\right\|<\frac12\xi,
\left(p'\round{\frac{ap'+b}q},P(z)\right)=1}\}
\displaybreak[0]\\
&\leqslant\#\{n:z\leqslant n\leqslant X,
p\mid\round{\frac{an+b}q},
\left\|\frac{an+b}q\right\|<\frac12\xi,
\left(n\round{\frac{an+b}q},P(z)\right)=1\}\\
&\qquad+O\big(\xi\pi(z)\big)\\
&=\CS(\widetilde\CA,z)+O\big(\xi\pi(z)\big),
\end{align*}

We now meet a sifting problem of dimension two. By Beta-Sieve theory we have
\begin{align*}
\CS(\CA_p,z)&\leqslant\CS(\widetilde\CA,z)+O\big(\xi\pi(z)\big)
\displaybreak[0]\\
&\leqslant\frac{X\xi}pV_2(z)\left(F_2(s_p')+O\left(
\left(\log\frac{X^{\theta_2}}p\right)^{-\frac16}\right)
\right)+\sum_{d\leqslant X^{\theta_2}/p}E(X;p,d)\\
&\qquad+O\big(\xi\pi(z)\big)
\displaybreak[0]\\
&=\frac{X\xi}pV_2(z)\left(F_2(s_p')+O\left(
(\log X)^{-\frac16}\right)
\right)+\sum_{d\leqslant X^{\theta_2}/p}E(X;p,d)
\displaybreak[0]\\
\end{align*}
and the last inequality holds because
$$
\xi\pi(z)\ll X\xi V_2(z)(\log X)^{-\frac16}.
$$

This completes the proof.
\end{proof}

\begin{lemma}\label{p:lem7-new}
If $\frac1a\leqslant\delta'\leqslant\frac ca\leq\theta_1$, let $w=X^{\delta'}$, then
\begin{align*}
\sum_{w\leqslant p\leq y}w_p\CS(\CA_p,z)
&\leqslant{\xi\pi(X)}
V(z)\left(ae^{-\gamma}\alpha
\int_{\delta'}^{\frac ca}\left(\frac1s-u\right)
F_2\left(a(\theta_2-s)\right)\mathrm ds + o(1)\right)\\
&\qquad +O\left(q\xi X^{\theta_2+\varepsilon}+\frac {X^{1+\varepsilon}}q\right).
\end{align*}
\end{lemma}
\begin{proof}
From lemma \ref{p:lem6-new} we obtain
\begin{align*}
\sum_{w\leqslant p\leq y}w_p\CS(\CA_p,z)&\leqslant X\xi V_2(z)
\left(\sum_{w\leqslant p\leq y}\frac{w_p}pF_2(s_p')
+O\left((\log X)^{-\frac16}\sum_{w\leqslant p\leq y}\frac1p\right)\right) \\
&\qquad+E_\CA(X;w,y)\\
&=X\xi V_2(z)\left(\sum_{w\leqslant p\leq y}\frac{w_p}pF_2(s_p')+o(1)\right)+E_\CA(X;w,y),\\
\end{align*}
where
\begin{align*}
E_\CA(X;w,y)&:=\sum_{w\leqslant p\leq y}w_p\sum_{d\leq\frac{X^{\theta_2}}{p}}E(X;p,d).
\end{align*}

Use the same method in Lemma \ref{p:lem5-new} to handle $\sum_{w\leqslant p\leq y}\frac{w_p}pF_2(s_p')$ and we obtain that
\begin{align*}
\sum_{w\leqslant p\leq y}\frac{w_p}pF_2(s_p')
&=\int_\delta^\frac1u\left(\frac1s-u\right)
F_2\left(\frac{4(\theta_2-s)}\alpha\right)\mathrm ds+o(1).
\end{align*}

As for $E_\CA(X;w,y)$, noting that for any $0<\theta_2<1$, and
for any $1\leqslant B\leqslant X^{\theta_2}/p$, we have
$$
\sum_{d\sim B}(k,d)=\sum_{c\mid k}\sum_{\substack{d\sim B\\
c=(k,d)}}c
=\sum_{c\mid k}c\sum_{\substack{d\sim Bc^{-1}\\
\left(d,kc^{-1}\right)=1}}1
\ll\sum_{c\mid k}B=B\tau(k),
$$
by Abel transformation,
$$
\sum_{d\sim B}\frac{(k,d)}d\ll\tau(k),
$$
which illustrates
\begin{equation}\label{eqn:4}
\sum_{d\leqslant X^{\theta_2}/p}\frac{(k,d)}d\ll\tau(k)\log X.
\end{equation}

Hence we conclude that
\begin{align*}
E_\CA(X;w,y)&\leq\sum_{w\leq p\leq y}\sum_{d\leq X^{\theta_2}/p}|E(X;p,d)|\\
&\ll\sum_{w\leq p\leq y}\sum_{d\leqslant X^{\theta_2}/p}\left(q\xi\tau(d)+
\frac X{pq}\sum_{d=d_1d_2}
\frac{(a,pd_2)(q,d_1)}{d_2d_1}\right)\\
&\ll q\xi \sum_{w\leq p\leq y}\frac {X^{\theta_2+\varepsilon}}p
+\frac Xq\sum_{w\leq p\leq y}\sum_{d_1d_2\leqslant X^{\theta_2}/p}
\frac{(a,pd_2)}{pd_2}\cdot\frac{(q,d_1)}{d_1}.
\displaybreak[0]\\
\end{align*}

Noticing that \eqref{eqn:4}, the lemma follows immediately.
\end{proof}

Thus we conclude our results above in a more general form:
Given $z=X^\alpha$, $y=X^\beta$ and $w=X^{\delta'}$, where $\alpha\leq\delta'\leq\beta\leq\theta_1$, then we have
\begin{align*}
\sum_{z\leq p<y}\CS\left(\CA_p,z\right)&=\sum_{z\leq p<w}\CS\left(\CA_p,z\right)+\sum_{w\leq p<y}\CS\left(\CA_p,z\right)\\
&\leq\sum_{z\leq p<w}\frac{\xi\pi(X)V(z)}{p}F_1(s_p)+\sum_{w\leq p<y}\frac{\xi XV_2(z)}{p}F_2(s_p')+\text{Error Term}\\
&=\xi\pi(X)V(z)\int_{z}^{w}F_1\left(\frac{\log(X^{\theta_1}/t)}{\log z}\right)d\sum_{p\geq t}\frac 1p\\
&\qquad+\xi XV_2(z)\int_{w}^{y}F_2\left(\frac{\log(X^{\theta_2}/t)}{\log z}\right)d\sum_{p\geq t}\frac 1p+\text{Error Term}\\
&=\xi\pi(X)V(z)\int_{z}^{w}\frac1{t\log t}F_1\left(\frac{\log(X^{\theta_1}/t)}{\log z}\right)dt\\
&\qquad+\xi XV_2(z)\int_{w}^{y}\frac1{t\log t}F_2\left(\frac{\log(X^{\theta_2}/t)}{\log z}\right)dt+\text{Error Term}\\
&=\xi\pi(X)V(z)\left(\int_{\alpha}^{\delta'}F_1\left(\frac{\theta_1-s}\alpha\right)\frac{ds}{s}+\frac{e^{-\gamma}}\alpha\int_{\delta'}^{\beta}
F_2\left(\frac{\theta_2-s}\alpha\right)\frac{ds}{s}+o(1)\right)\\
\end{align*}
Similarly, we have
\begin{align*}
\sum_{z\leq p<y}w_p\CS\left(\CA_p,z\right)
&\leq\xi\pi(X)V(z)\int_{\alpha}^{\delta'}\left(\frac1s-u\right)F_1\left(\frac{\theta_1-s}\alpha\right)ds\\
&\qquad+\xi\pi(X)V(z)\left(\frac{e^{-\gamma}}\alpha\int_{\delta'}^{\beta}
\left(\frac1s-u\right)F_2\left(\frac{\theta_2-s}\alpha\right)ds+o(1)\right).\\
\end{align*}
As shown later in this paper, we can optimize $\delta'$ to make the upper bounds of $\sum_{z\leq p<y}\CS\left(\CA_p,z\right)$ or $\sum_{z\leq p<y}w_p\CS\left(\CA_p,z\right)$ achieve their minimal value, where
$$
\delta'=\delta_0=\theta_2-\frac12\left(A_3
-\sqrt{A_3^2-4A_3\left(\theta_2-\theta_1\right)}\right),
$$
which is actually very close to $\theta_1$.
If we take $a\theta_1=6$, which is a simple but effective choice, then the computations from \cite{Lab79} tell us that
$$
b<\frac{\log(1 + e^{24B}) + D - \log6}{B} - \frac{18}{1 + e^{-24B}}\approx4.2,
$$
while $a\geq18$ by Theorem \ref{p:thm4}. Hence $\frac ba<0.24<\delta_0$ if $\rho\geq 15$, since actually we can take
$$
\delta_0=\frac23-\rho-\frac12\left(A_3
-\sqrt{A_3^2-\frac{4A_3}3}\right)\pm10^{-10}.
$$

Therefore, we can only use a 2-dimensional sieve to the last term in Lemma \ref{pp}.
\begin{lemma}\label{p:lem2}
We have
$$
\CS\left(\CA^*,X^{\frac12-\eta}\right)\leqslant
\big(4\mathfrak{I}(\rho)+o(1)\big)\frac{{\xi\pi(X)}}{\log X}
+O\left(X^{\varepsilon}\sum_{r\leqslant X^\nu}|R_r^*(\beta)|\right),
$$
where
\begin{equation}\label{eqn:5}
\sum_{r\leqslant X^\nu}|R_r^*(\beta)|\ll{\xi\pi(X)}X^{-\frac\eta2},
\end{equation}
with $\nu=\frac{1-\beta}2-\rho-2\eta$
and $\mathfrak{I}(\rho)$ is defined by
\begin{equation}\label{eqn:p}
\mathfrak{I}(\rho)
:=\int_{\frac\alpha 4}^{\frac14}\frac{\mathrm du_1}{u_1(1-u_1-2\rho)}
\int_{u_1}^{\frac{1-u_1}3}\frac{\mathrm du_2}{u_2}
\int_{u_2}^{\frac{1-u_1-u_2}2}\frac{\mathrm du_3}{u_3(1-u_1-u_2-u_3)}.
\end{equation}
\end{lemma}
\begin{proof}
This follows from \cite{4:HR74}, Theorem 8.3 and \cite{1:Har84}.
\end{proof}
\begin{remark}
We shall use \eqref{eqn:5} to give some restrictions in Theorem \ref{p:thm4}.
\end{remark}

\begin{proof}[Proof of theorem \ref{p:thm1}]

We have
$$
\sum_{p\geqslant X^{\frac1a}}\sum_{h\in\CA_{p^2}}1
\ll\sum_{p\geqslant X^{\frac1a}}\frac{\pi(x)\xi}{p^2}
\ll\frac{\pi(x)\xi}{X^{\frac1a}}
\ll X^{1-\eta}\xi
=o\left({\xi\pi(X)}V(z)\right).
$$
It comes from lemma \ref{p:lem4-new},
lemma \ref{p:lem5-new} and lemma \ref{p:lem7-new} that
\begin{align*}
\lambda^{-1}\CW(\CA,u,\lambda)&\geq\lambda^{-1}\CS(\CA,z)-\CS_{\widetilde{w}}(\CA)+o\left({\xi\pi(X)}V(z)\right)\\
&=\CW_1(\CA,u,\lambda)-\CW_2(\CA,u,\lambda)+o\left({\xi\pi(X)}V(z)\right),
\end{align*}
where
\begin{align*}
\CW_1(\CA,u,\lambda)&:=\lambda^{-1}\CS(\CA,z)-(c-b)\sum_{X^{\frac 1a}\leq p<X^{\frac ba}}\CS\left(\CA_p, X^{\frac1a}\right)\\
&\qquad-a\int_{\frac1a}^{\frac{b+1}{2a}}\left(\sum_{X^{s}\leq p\leq X^{\frac{b+1}{a}-s}}\CS\left(\CA_p, X^{s}\right)ds\right)-c\sum_{X^{\frac ba}\leq p<X^{\delta}}w_p\CS\left(\CA_p, X^{\frac1a}\right)\\
&\qquad-c\sum_{X^{\frac 1a}\leq p<X^{\frac {b+1}{2a}}}\left(\frac{b+1}{c}-\frac{2u\log p}{\log X}\right)\CS\left(\CA_p, p\right)\\
&\geq\xi\pi(X)V(z)\{(5c-a)f_1(a\theta_1)-\int_{\frac1{a\theta_1}}^{\frac{b+1}{2a\theta_1}}\left(\int_{s}^{\frac{b+1}{a\theta_1}-s}F_1\left(\frac{1-t}s\right)
\frac{dt}t\right)\frac{ds}s\\
&\qquad-(c-b)\int_{\frac1{a\theta_1}}^{\frac b{a\theta_1}}F_1(a\theta_1(1-s))\frac{ds}s-\int_{\frac b{a\theta_1}}^{\delta}(\frac cs-a\theta_1)F_1\left(a\theta_1(1-s)\right)ds\\
&\qquad-\int_{\frac1{a\theta_1}}^{\frac{b+1}{2a\theta_1}}\left(\frac{b+1}{a\theta_1}-2s\right)F_1\left(\frac{1-s}s\right)\frac{ds}{s^2}\}+o(1)\\
&\geq\xi\pi(X)V(z)\{(5c-a)f_1(\vartheta)-\int_{\frac 1{\vartheta}}^{\frac{b+1}{2\vartheta}}\left(\int_{s}^{\frac{b+1}{\vartheta}-s}F_1\left(\frac{1-t}s\right)\frac{dt}t\right)\frac{ds}s\\
&\qquad-(c-b)\int_{\frac1{\vartheta}}^{\frac b\vartheta}F_1(\vartheta(1-s))\frac{ds}s-\int_{\frac b\vartheta}^{\frac c\vartheta}(\frac cs-\vartheta)F_1\left(\vartheta(1-s)\right)ds+o(1)\\
&\qquad-\int_{\frac1\vartheta}^{\frac{b+1}{2\vartheta}}\left(\frac{b+1}{\vartheta}-2s\right)F_1\left(\frac{1-s}s\right)\frac{ds}{s^2}
+\int_{\delta}^{\frac c\vartheta}(\frac cs-\vartheta)F_1\left(\vartheta(1-s)\right)ds\};
\end{align*}
where $\vartheta=a\theta_1$, $\delta\in[\frac b\vartheta, \frac c\vartheta]$, and
\begin{align*}
\CW_2(\CA,u,\lambda)&:=c\sum_{X^{\delta}\leq p<X^{\frac ca}}w_p\CS\left(\CA_p, X^{\frac1a}\right)\\
&\leq ae^{-\gamma}\xi\pi(X)V(z)\left(\int_{\delta}^{\frac c\vartheta}\left(\frac cs-\vartheta\right)
F_2\left(a\theta_2-\vartheta s\right)\mathrm ds
+O\left( q\xi X^{\theta_2+\varepsilon}
+\frac {X^{1+\varepsilon}}q\right)\right).
\end{align*}
To be admissible, $\theta_2$ can be taken to be any number smaller than $\frac23-\rho$ since $q\asymp q^{\frac{1}{3}+\rho+\eta}.$ Take $\theta_2= \frac {2}{3}-\rho-\varepsilon$ and $\frac{\theta_2}{\theta_1}\theta\rightarrow\theta:=\frac{\frac23-\rho}{\frac13-\rho}$ as $\varepsilon\rightarrow0^{+},$
thus by continuity and Lemma \ref{p:lem2} we have, when $\varepsilon$ is sufficiently small,
\begin{align*}
\mathcal{J}(\lambda)\geq \frac{ae^{-\gamma}\lambda\xi\pi(X)}{\log X}\mathcal{H}(\vartheta,b,c),
\end{align*}
where
\begin{align*}
\mathcal{H}(\vartheta,b,c)=\mathcal{H}_{\delta}(\vartheta,b,c)&:=(5c-a)f_1(\vartheta)-\int_{\frac 1{\vartheta}}^{\frac{b+1}{2\vartheta}}\left(\int_{s}^{\frac{b+1}{\vartheta}-s}F_1\left(\frac{1-t}s\right)\frac{dt}t\right)\frac{ds}s\\
&\qquad-(c-b)\int_{\frac1{\vartheta}}^{\frac b\vartheta}F_1(\vartheta(1-s))\frac{ds}s-\int_{\frac b\vartheta}^{\frac c\vartheta}(\frac cs-\vartheta)F_1\left(\vartheta(1-s)\right)ds+o(1)\\
&\qquad-\int_{\frac1\vartheta}^{\frac{b+1}{2\vartheta}}\left(\frac{b+1}{\vartheta}-2s\right)F_1\left(\frac{1-s}s\right)\frac{ds}{s^2}
+\int_{\delta}^{\frac c\vartheta}(\frac cs-\vartheta)F_1\left(\vartheta(1-s)\right)ds\\
&\qquad- ae^{-\gamma}\left(\int_{\delta}^{\frac c\vartheta}\left(\frac cs-\vartheta\right)
F_2\left(a\theta_2-\vartheta s\right)\mathrm ds\right)-\frac{4e^{\gamma}c}a\mathfrak{I}(\rho).
\end{align*}
Then
\begin{align*}
\mathcal{H}_{c}'(\vartheta,b,c)=2f_1(\vartheta)-\int_{\frac1\vartheta}^{\delta}F_1\left(\vartheta(1-s)\right)\frac{ds}s-ae^{-\gamma}\int_{\delta}^{\frac c\vartheta}F_2\left(\vartheta(\theta-s)\right)\frac{ds}s-\frac{4e^{\gamma}}a\mathfrak{I}(\rho),
\end{align*}
and
\begin{align*}
\mathcal{H}_{b}'(\vartheta,b,c)&=-f_1(\vartheta)+\int_{\frac1\vartheta}^{\frac b\vartheta}F_1\left(\vartheta(1-s)\right)\frac{ds}s-\frac1\vartheta\int_{\frac1\vartheta}^{\frac{b+1}{2\vartheta}}F_1\left(\frac{1-s}s\right)\frac{ds}{s^2}\\
&\qquad-\int_{\frac1\vartheta}^{\frac{b+1}{2\vartheta}}F_1\left(\frac{\vartheta s+\vartheta-1-b}{\vartheta s}\right)\frac{ds}{s(b+1-\vartheta s)}.
\end{align*}
Assume that $\vartheta\geq 4$ and $b\geq\vartheta-3$, then we have
\begin{align*}
\mathcal{H}_{bc}''(\vartheta,b,c)&=0;\\
\mathcal{H}_{b^2}''(\vartheta,b,c)&=\frac1b F_1(\vartheta-b)-\left(-\frac1{2\vartheta-1-b}+\frac1{\vartheta-b}-\frac1{b+1}+\frac1b\right)\\
&\qquad-\frac{2}{(b+1)^2}F_1\left(\frac{2\vartheta}{b+1}-1\right)=0;\\
\mathcal{H}_{c^2}''(\vartheta,b,c)&=-\frac{ae^{-\gamma}F_2(\vartheta\theta-c)}{c}.
\end{align*}

Thus we can write $\mathcal{H}(\vartheta,b,c)$ as
\begin{align*}
\mathcal{H}(\vartheta,b,c)=2e^{\gamma}\left(A(\vartheta)b+B(\vartheta)c+D(\vartheta)+\mathfrak{F}(\vartheta,c)\right).
\end{align*}
where
\begin{align*}
\mathfrak{F}(\vartheta,\vartheta\theta)=0~\ \text{and}~\
\mathfrak{F}'(\vartheta,c)=-\frac a{2e^{2\gamma}}\int_{\delta}^{\frac c\vartheta}F_2\left(\vartheta(\theta-s)\right)\frac{ds}s,
\end{align*}
and $A(\vartheta)$, $B(\vartheta)$, $D(\vartheta)$ are determined by
\begin{align*}
2e^{\gamma}\left(A(\vartheta)+B(\vartheta)\right)
&=f_1(\vartheta)-\frac1\vartheta\int_{\frac1\vartheta}^{\frac{b+1}{2\vartheta}}F_1\left(\frac{1-s}s\right)\frac{ds}{s^2}-\frac{4e^{\gamma}}{a}\mathfrak{I}(\rho)\\
&\qquad-\int_{\frac1\vartheta}^{\frac{b+1}{2\vartheta}}F_1\left(\frac{\vartheta s+\vartheta-1-b}{\vartheta s}\right)\frac{ds}{s(b+1-\vartheta s)}\\
&=f_1(\vartheta)-\frac{2e^{\gamma}}{\vartheta}\log\frac{\delta(2\vartheta-\vartheta\delta-1)}{(1-\delta)(\vartheta\delta+1)}
-f_1(\vartheta)+\frac{2}{\vartheta\delta+1}f_1\left(\frac{2\vartheta}{\vartheta\delta+1}\right)-\frac{4e^{\gamma}}{a}\mathfrak{I}(\rho)\\
&=\frac{2e^{\gamma}}{\vartheta}\log\frac{1-\delta}{\delta}-\frac{4e^{\gamma}}{a}\mathfrak{I}(\rho);
\end{align*}
Also we have, by direct computation,
\begin{align*}
B(\vartheta)&=e^{-\gamma}f_1(\vartheta)-\frac1{2e^{\gamma}}\int_{\frac1{\vartheta}}^{\delta}
F_1\left(\vartheta(1-s)\right)\frac{ds}{s}
-\frac{2}{a}\mathfrak{I}(\rho),
\end{align*}
so
\begin{align*}
A(\vartheta)&=-e^{-\gamma}f_1(\vartheta)+\frac1{2e^{\gamma}}\int_{\frac1{\vartheta}}^{\delta}
F_1\left(\vartheta(1-s)\right)\frac{ds}{s}+\frac{1}{\vartheta}\log\frac{1-\delta}{\delta};\\
D(\vartheta)&=\frac1{2e^{\gamma}}\mathcal{H}(\vartheta,\vartheta\theta,\vartheta\theta)-\delta\log\frac{1-\delta}{\delta}+\frac{2\vartheta\theta}{a}\mathfrak{I}(\rho).
\end{align*}
Thus , then by the continuity of $F_2$ we obtain theorem \ref{p:thm1}.
\end{proof}

\section{Proof of Theorem \ref{p:main}}
It is obvious that by Corollary \ref{4} and Theorem \ref{p:thm1} we have:
\begin{thm}\label{p:thm3}
The restriction from the main terms is given by
$$
\begin{cases}
\displaystyle 1\leq b\leq c\leq a=\frac{\vartheta}{\theta_1} \\[0.5em]
b=1 ~\ \text{or}~\ a\geq 3c+b+1,\quad \text{if}~\ b\geq 3 \\[0.5em]
\frac b{\vartheta}\leq\delta_0\leq\frac c{\vartheta}\\[0.5em]
\displaystyle \max_{\frac b{\vartheta}\leq\delta\leq\frac c{\vartheta}}\mathcal{H}_{\delta}(\vartheta,b,c)>0.
\end{cases}
$$
where
$\mathcal{H}_{\delta_0}(\vartheta,b,c)$ is defined by \eqref{x:10} with $F_2$ defined as before Theorem \ref{p:thm1}.
\end{thm}

\begin{thm}\label{p:thm4}
The restrictions from the error terms are given as the
following inequation systems:
$$
\begin{cases}
0<\rho<\min\left\{\frac16,\frac1a\right\}, \\
\theta_1+\rho<\frac13,~\theta_1>0.
\end{cases}
$$
\end{thm}
\begin{proof}
In Corollary \ref{p:cor13-new} and Lemma \ref{p:lem6} above, where we show that
$$
\sum_{r\leqslant X^\alpha}\frac1r
\max_{N\leqslant X}\sum_{l=1}^{rY}
\left|\sum_{n\leqslant N}\Lambda(n)
e\left(\frac{\alpha nl}{rq}\right)\right|\ll\pi(X)X^{-\eta},
$$
with $Y\asymp X^{\rho+\eta}$, we have to make sure that all the parameters
satisfy the assumptions of those lemmas.

Divide the intervals into dyadic segments and thus we have the
following estimation:
\begin{align*}
\sum_i\sum_{i_j}X^\varepsilon\left(
X^{\frac23}T_i+X^{\frac{11}{12}}\left(
\frac{T_i}{R_i}\right)^{\frac12}\right)
&\ll X^{2\varepsilon}X^{\frac23+\theta_1+\rho+\eta}
+X^{\frac{11}{12}}\sum_i\sum_{i_j}X^{\frac{\rho+\eta}2} \\
&\ll X^{2\varepsilon}X^{\frac23+\theta_1+\rho+\eta}
+X^{\frac{11}{12}+\varepsilon+\frac{\rho+\eta}2},
\end{align*}
where $L_{i_j}\leqslant R_iY\ll R_iX^{\rho+\eta}$,
$T_i\ll R_iX^{\rho+\eta}$
and for simplicity
we omit the precise range of $i$ and $j$, actually, only the bound
$i,j\ll\log X$ matters.

Therefore, we get our restrictions
as below:
$$
\begin{cases}
2\varepsilon+\frac23+\theta_1+\rho+\eta<1-\eta, \\
\frac{11}{12}+\varepsilon+\frac{\rho+\eta}2<1-\eta,
\end{cases}
$$
i.e.
$$
\begin{cases}
\theta_1+\rho<\frac13, \\
\rho<\frac16.
\end{cases}
$$

Now let's consider another estimation from \eqref{eqn:5}.
By assumption, we have $X^{\rho+\beta+\eta}
<X^{\rho+\frac14+\eta}<q<X^{\frac34-\eta}<X^{1-\beta-\eta}$.
Additionally, by Lemma \ref{p:lem6}, there should be
$$
\xi X^{1+3\varepsilon}\sum_i\sum_{i_j}\left(L_{i_j}+\frac{R_i}{X^\beta}\right)
\left(X^{\frac{\beta-1}2}+X^{-\beta}R_i^{-\frac12}
\left(L_{i_j}+\frac{R_i}{X^\beta}\right)^{-\frac12}\right)
\ll{\xi\pi(X)}X^{-\eta}.
$$
While
\begin{multline*}
\xi X^{\frac{\beta+1}2+3\varepsilon}\sum_i\sum_{i_j}
\left(L_{i_j}+\frac{R_i}{X^\beta}\right) \\
\ll\xi X^{\frac{\beta+1}2+4\varepsilon}
\left(X^{\frac{1-\beta}2-\rho-2\eta}
+X^{-\beta}X^{\frac{1-\beta}2-2\eta}\right)
\ll\xi X^{1-\rho-2\eta+4\varepsilon}+\xi X^{1-\beta-2\eta},
\end{multline*}
and
\begin{multline*}
\xi X^{1+3\varepsilon}\sum_i\sum_{i_j}\left(X^{-\frac\beta2}R_i^{-\frac12}
\left(L_{i_j}+\frac{R_i}{X^\beta}\right)^{\frac12}\right) \\
\ll\xi X^{1+3\varepsilon}\sum_i\sum_{i_j}\left(X^{-\frac\beta2}R_i^{-\frac12}
L_{i_j}^{\frac12}+X^{-\beta}\right)
\ll\xi X^{1+4\varepsilon+\frac{\rho+\eta}2-\frac\beta2},
\end{multline*}
so it suffices to have the restriction:
$1+4\varepsilon+\frac{\rho+\eta}2-\frac\beta2<1-\frac\eta2-\varepsilon$,
which could be deduced by the condition:
$\beta>\rho\Leftarrow\frac1a>\rho$.
This completes the proof.
\end{proof}

Combine all the restrictions from Theorem \ref{p:thm3} and
Theorem \ref{p:thm4}. Take $\theta_1=\frac13-\rho-10^{-12}$, then insert this into the above conditions,
with the help of the software \emph{Mathematica 9},
we find that
$\rho=\frac1{118}$ satisfies the restrictions above, when $b=1,$ $c=3.98,$ $\vartheta=4.07$ (hence $a\approx12.5285$), noting that $a$ slight larger than $\frac4\alpha$ in \cite{1:Har84}.
Thus we have proven that
there are infinitely many solutions of
$$
|\lambda_0+\lambda_1p+\lambda_2P_3|<p^{-\frac1{118}}.
$$
\begin{remark}
We thus see that in our situation Laborde's weight is not better than Richert's weight because of the effect from $\CS(\CA(\beta)^*,X^{\frac12-\eta}),$ since $\frac{\mathfrak{I}(\rho)}a=a\mathfrak{I}(\rho,a)$ grows faster than $f_1(\vartheta)$ when $\vartheta\geq4$. When $b>1,$ which forces that $a\geq3c+b+1>5,$ the contribution of $\CS(\CA(\beta)^*,X^{\frac12-\eta})$ is just too large for our purpose. If we just take $\delta=\alpha$ as Harman did in \cite{1:Har84}, then by optimizing the parameters directly we have $\tau<\frac1{146}$ and we can take $\tau=\frac1{147}.$
\end{remark}

\section{Estimates for exponential sums II}\label{s:15}
\begin{lemma}[\cite{BF92}]\label{w}
For any $\iota\in[0,1]$, let
$$
f_{h,\iota,\varsigma}(x):=h\left(x+\iota\right)^{\gamma}+\varsigma x,
$$
where $h\in \BN$ and $\varsigma$ is an arbitrary constant. Take $\sigma$ satisfying the restriction  $\sigma<\frac{9\gamma-8}{12}$. Then any sufficiently small $\eta>0$, we have
$$
\min\left\{1, \frac{X^{1-\gamma}}H\right\}\sum_{h\sim H}\left|\sum_{n\sim X}\Lambda(n)e\left(f_{h,\iota,\varsigma}(n)\right)\right|\ll_{\eta}X^{1-\sigma-3\eta},
$$
where $H\leq X^{1-\gamma+\sigma+\varepsilon}$.
\end{lemma}
\begin{remark}
We should point out that the $O$-constant is independent of $\iota$ and $\varsigma$, namely, it's uniform for $\varsigma$, because only the behavior of $f_{h,\varsigma}''(x)$ is used when handling sums of booth Type I and Type II, after using Heath-Brown's identity (see \cite{HB83}). This is a critical property as we will see in our situation we actually need to bound a mean estimate of the form
$$
\sum_{d\sim D}\sum_{l\sim L}|b_l|\sum_{h\sim H}\frac 1h\max_{N\leq X}\left|\sum_{n\sim N}\Lambda(n)e\left(h(n+\iota)^{\gamma}+\frac{aln}{qd}\right)\right|.
$$
We are showing the level of distribution is $\theta_3=\frac{9\gamma-8}{12}-\rho$.
\end{remark}

In this section we aim to prove the following lemma:
\begin{lemma}\label{error£º2}
For $c\in \left(1, \frac{755}{662}\right)$, $\theta_3=\frac{1+9(\gamma-1)}{12}-\rho,$ we have
\begin{equation}\label{E:3}
\sum_{d\leq X^{\theta_3}}\frac{\xi}d\max_{N\leq X}\sum_{l=1}^{d Y}\left|\sum_{\substack{n\leq N\\n\in \CP}}\Lambda(n)e\left(\frac{anl}{dq}\right)\right|\ll
\xi\pi_{c}(X^{\gamma})X^{-\eta},
\end{equation}
where $\gamma=\frac1c$ and $\CP:=\left\{\floor{n^c}: n\in \BN\right\}$.
\end{lemma}
\begin{proof}
It is clearly that $p=\floor{n^c}$ if and only if there exists a nonnegative $\nu<1$ such that $n^c=p+\nu$, which, by a direct check, is equivalent to
$$
\floor{-p^\gamma}-\floor{-(p+1)^\gamma}=1,
$$
where $\gamma$ is taken to be the inverse of $c$ traditionally.

Hence we can take $\phi(n):=\floor{-n^\gamma}-\floor{-(n+1)^\gamma}$ to be a characteristic function of $\CP$, and thus for any $N\leq X$, we have
\begin{align*}
\sum_{l=1}^{d Y}\left|\sum_{\substack{n\leq N\\n\in \CP}}\Lambda(n)e\left(\frac{anl}{dq}\right)\right|&=\sum_{l=1}^{d Y}\left|\sum_{n\leq N}\phi(n)\Lambda(n)e\left(\frac{anl}{dq}\right)\right|\leq\CE_1(N,d)+\CE_2(N,d),
\end{align*}
where
\begin{equation}\label{E:1}
\CE_1(N,d):=\sum_{l=1}^{d Y}\left|\sum_{n\leq N}\left((n+1)^\gamma-n^\gamma\right)\Lambda(n)e\left(\frac{anl}{dq}\right)\right|,
\end{equation}
and
\begin{align*}
\CE_2(N,d):=\sum_{l=1}^{d Y}\left|\sum_{n\leq N}\left(\{-n^\gamma\}-\{-(n+1)^\gamma\}\right)\Lambda(n)e\left(\frac{anl}{dq}\right)\right|.
\end{align*}
We will see later that $\CE_1(N,d)$ and $\CE_2(N,d)$ are different types of exponential sums, and the former is algebraic, while the latter is analytic. Hence we use different methods to handle them respectively.
\begin{description}
\item[Estimate of $\CE_1(N,d)$] Write $\CE_1(N,d)$ in an integral form and integral by parts we have
\begin{align*}
\CE_1(N,d)&:=\sum_{l=1}^{d Y}\left|\int_{1}^{ N}\left((t+1)^\gamma-t^\gamma\right)d\sum_{n\leq t}\Lambda(n)e\left(\frac{anl}{dq}\right)\right|\\
&=\int_{1}^{N}\left((t+1)^\gamma-t^\gamma\right)d\left(\sum_{l=1}^{d Y}c_l\sum_{n\leq t}\Lambda(n)e\left(\frac{anl}{dq}\right)\right)\\
&\leq\int_{1}^{N}\max_{T\leq N}\left|\sum_{l=1}^{d Y}c_l\sum_{n\leq t}\Lambda(n)e\left(\frac{anl}{dq}\right)\right|\left((t+1)^{\gamma-1}-t^{\gamma-1}+O(\frac1N)\right)dt\\
&\ll\max_{T\leq N}\sum_{l=1}^{d Y}\left|\sum_{n\leq T}\Lambda(n)e\left(\frac{anl}{dq}\right)\right|,
\end{align*}
where $c_l=e^{i\theta_l}$, here $\theta_l$ is the principle argument of the inner sum in (\ref{E:1}).
Thus by Lemma \ref{p:lem6} we have $YX^{\frac23+\theta_3}\ll\pi_{c}(X^{\gamma})X^{-\eta},$ deducing that
$$
\theta_3\leq\gamma-\frac23-\rho-\varepsilon.
$$
\item[Estimate of $\CE_2(N,d)$] Take $\eta=3\varepsilon$. By Lemma \ref{p:lem1-new} we have
\begin{align*}
\CE_2(N,d)&=\sum_{l=1}^{d Y}c_l\sum_{n\leq N}\left(\sum_{1\leq|h|\leq X^{1-\gamma+\sigma+\varepsilon}}\frac{e(h(n+1)^\gamma-e(hn^\gamma))}{2\pi ih}\right)\Lambda(n)e\left(\frac{anl}{dq}\right)\\
&\qquad+O\left(X^{\gamma-1-\sigma}\sum_{l=1}^{d Y}\sum_{n\leq N}\Lambda(n)\right)\\
&=\sum_{l=1}^{d Y}c_l\sum_{1\leq|h|\leq X^{1-\gamma+\sigma+\varepsilon}}\sum_{n\leq N}\Lambda(n)\frac{\CE_2^{1}(N,d)-\CE_2^{0}(N,d)}{2\pi ih}+O\left(dYX^{\gamma-\sigma+\eta}\right).
\end{align*}
where
$$
\CE_2^{\iota}(N,d):=\sum_{l=1}^{d Y}c_l\sum_{1\leq|h|\leq X^{2+\varepsilon}}\frac1h\sum_{n\leq N}\Lambda(n)e\left(h(n+\iota)^\gamma+\frac{anl}{dq}\right),
$$
for $\iota\in\{0,1\}$. We split the summation range into  dyadic segments, a typical one is
\begin{align*}
\CE_{2_j}^{\iota}(N,d)&:=\sum_{l=1}^{d Y}c_l\sum_{h\sim H}\frac1h\sum_{n\leq N}\Lambda(n)e\left(h(n+\iota)^\gamma+\frac{anl}{dq}\right)\\
&\ll\sum_{l=1}^{d Y}\frac1H\sum_{h\sim H}\left|\sum_{n\leq N}\Lambda(n)e\left(h(n+\iota)^\gamma+\frac{anl}{dq}\right)\right|,
\end{align*}
where $H$ is of the form $2^jX^{1-\gamma}$, and $j\ll\log X$ since $H\leq X^{1-\gamma+\sigma+\varepsilon}$.
Hence by Lemma \ref{w} we have
\begin{align*}
\sum_{d\leq X^{\theta_3}}\frac{\xi}d\max_{N\leq X}\left|\CE_2(N,d)\right|&\ll
\sum_{\iota\in\{0,1\}}\sum_{j\ll\log X}\sum_{d\leq X^{\theta_3}}\frac{\xi}d\max_{N\leq X}\left|\CE_{2_j}^{\iota}(N,d)\right|\\
&\ll \sum_{d\leq X^{\theta_3}}\xi YX^{\gamma-\sigma-2\eta}\ll \xi^2\pi_c(X^\gamma)X^{\theta_3-\sigma-\eta}.
\end{align*}
So it suffices to take $\theta_3\leq\frac{1+9(\gamma-1)}{12}-\rho$.
\end{description}
Combining the above discussion we thus obtain (\ref{E:3}).
\end{proof}
\section{Proof of Theorem 2}
Denote by
\begin{align*}
\hat{\CB}&:=\left\{\round{\frac{b+pa}q}:
p\leqslant X, p\in \CP,\left\|\frac{b+pa}q\right\|<\frac\xi2
\right\}, \\
\end{align*}
where $\CP:=\left\{\floor{n^c}: n\in \BN\right\}$.
By taking $M\asymp\frac{dX^{\eta}}{\xi}$ in Lemma \ref{5} we have for and $d\in\BN$
\begin{align*}
\#\CB_d&=\sum_{\substack{p\leqslant X, p\in \CP\\
d\mid\round{\frac{ap+b}q}\\
\left\|\frac{ap+b}q\right\|<\frac\xi2}}1
=\sum_{\substack{p\leqslant X, p\in \CP\\
\left\|\frac{ap+b}{dq}\right\|<\frac\xi{2d}}}1
=\sum_{p\leqslant X, p\in \CP}\chi\left(\frac{ap+b}{dq}\right)
\displaybreak[0]\\
&=\frac{\pi_c(X^\gamma)\xi}d+E(\CB_d)+O\left(\frac{{\xi\pi_c(X^\gamma)}X^{-\eta}}d\right),\\
\end{align*}
where
\begin{align*}
\sum_{p\leqslant X, p\in \CP}\sum_{1\leq|l|\leq M}
a_le\left(\frac{(ap+b)l}{qd}\right)\leq E(\CB_d)\leq\sum_{p\leqslant X, p\in \CP}\sum_{1\leq|l|\leq M}
b_le\left(\frac{(ap+b)l}{qd}\right)\\
\end{align*}
with $
|a_l|+|b_l|\ll\frac\xi d,  \forall 1\leq|l|\leq M.
$

As shown in Lemma \ref{p:lem4-new}, by partial summation we have\begin{align*}
E(\CB_d)&\ll\max_{N\leqslant X}\frac1 {\log X}\sum_{1\leq|l|\leq M}\left(|a_l|+|b_l|\right)\left|\sum_{n\leqslant N, p\in \CP}\Lambda(n)e\left(\frac{anl}{qd}\right)
\right|\\
&\ll\max_{N\leqslant X}\frac\xi d\sum_{l=1}^{dY}\left|
\sum_{n\leqslant N, p\in \CP}\Lambda(n)e\left(\frac{anl}{qd}\right)
\right|,\\
\end{align*}
so the density function of sequence $\CB$ is $g_3(d)=\frac1d$, and the corresponding level of distribution $\theta_3$ can be taken to be $\frac{1+9(\gamma-1)}{12}-\rho$.

Since the level here is quite small, there might be little room for other sieve techniques. Thus we choose to use Laborde's results to deal with $\CB$ directly.
\begin{lemma}
There are infinitely many $P_{r}$ in $\CB$ if $$\frac{c}{\theta_3}\leq r-0.144.$$
\end{lemma}
\begin{proof}
This is essentially Theorem 3 of \cite{Lab79}. However, the upper bound for $\Lambda$ there can actually be taken to be 0.144, since
$$
\frac{\log6-B-D}{6B}-\frac{\log(1+e^{-78B})}{6B}\approx0.144002.
$$
So we can take $0.144$ rather than $0.145$ in the statement of Laborde's theorem. This leads us to take $\rho=\frac1{180}$, otherwise, we can only take $\rho=\frac1{181}.$
\end{proof}
Take $r=13$ and Theorem \ref{p:main2} follows immediately.
\begin{remark}
Similarly, we can also use 2-dimensional sieve to sharp the range of $\rho$.
\end{remark}


\begin{thebibliography}{99}
\bibitem{BF92}
A. Balog, J. Friedlander. \emph{A hybrid of theorems of Vinogradov and Piatetski-Shapiro}. Pacific J. Math. 156 (1992), 45-62.
\bibitem{BGS15}
W. D. Banks, V. Z. Guo, I. E. Shparlinski. \emph{Some arithematic properities of numbers of the form $\floor{p^c}$}. Preprint 2015 (arXiv: 1508.03281v2).
\bibitem{FI10}
J. Friedlander, H. Iwaniec. \emph{Opera de cribro}.
Vol. 57. American Mathematical Soc., 2010.
\bibitem{2:Gal67}
P. X. Gallagher. \emph{The large sieve}. Mathematika, 14 (1967), 14--20.
\bibitem{4:HR74}
H. Halberstam, H. E. Richert. \emph{Sieve methods}. Academic
Press, London, 1974.
\bibitem{1:Har84}
G. Harman. \emph{Diophantine approximation with a prime
and an almost-prime}. J. London Math. Soc. (2), 29 (1984), 13--22.
\bibitem{HB83}
D. R. Heath-Brown. \emph{The Pjatecki-$\check{S}apiro$ prime number theorem}. J. Number Theory, 16 (1983) 242-266.
\bibitem{IL81}
H. Iwaniec, M. Laborde. \emph{$P_2$ in short intervals}. Ann. Inst. Fourier, Grenoble. 31, 4 (1981), 37-56.
\bibitem{Irv 15}
A. J. Irving, \emph{Almost-prime values of polynomials at prime arguments}.
 Bull. London Math. Soc. (2015) 47 (4): 593-606.
\bibitem{Lab79}
M. Laborde. \emph{Buchstab's sifting weights}. Mathematika, 26 (1979), 250-257.
\bibitem{Mont78}
H. L. Montgomery. \emph{The analytic principle of the large sieve}, Bull. Amer. Math. Soc, 84 (1978), 547-567.
\bibitem{7:Ram77}
K. Ramachandra. \emph{Two remarks in prime number theory}. Bull.
Soc. Math. France, 105 (1977), 433--437.
\bibitem{8:Sri82}
S. Srinivasan. \emph{A note on $|\alpha p-q|$}. Acta Arith., 41 (1982), 15--18.
\bibitem{Vau76}
R. C. Vaughan. \emph{Diophantine approximation by prime numbers}.
Proc. London Math. Soc. (3), 33 (1976), 177--192.
\bibitem{Vau77}
R. C. Vaughan. \emph{On the distribution of $\alpha p$ modulo $1$}.
Mathematika, 24 (1977), 135--141.
\bibitem{Vau80}
R. C. Vaughan, \emph{ An elementary method in prime number theory}. Acta Arith, 37 (1980), 111-115.
\bibitem{V54}
I. M. Vinogradov, \emph{The method of trigonomrtrical sums in the theory of numbers}, translated revised and annotated by K. F. Roth and Anne Davenport, Interscience Piblishers, London and New York, 1954.
\end{thebibliography}
\end{document}